\documentclass[leqno,11pt]{scrartcl}
\usepackage[utf8]{inputenc}

\usepackage{amsmath,amssymb,amsthm}
\usepackage{amscd}
\usepackage{setspace} %%% Einstellung des Zeilenvorschubs mit \spacing{ } %%%
\usepackage{enumerate}
\usepackage{array}
\usepackage{tabularx}
\usepackage{multirow}
\usepackage{booktabs}
\usepackage[hidelinks]{hyperref}
\usepackage{hyperref}
\usepackage[german, english]{babel}

\theoremstyle{definition}
\newtheorem{definition}{Definition}
\newtheorem{remark}{Remark}

\theoremstyle{plain}
\newtheorem{theorem}{Theorem}

\newtheorem{lemma}[definition]{Lemma}
\newtheorem{corollary}{Corollary}

\theoremstyle{remark}

\newcommand{\C}{\mathbb{C}}

\newcommand{\N}{\mathbb{N}}

\newcommand{\Q}{\mathbb{Q}}
\newcommand{\R}{\mathbb{R}}
\newcommand{\Z}{\mathbb{Z}}

\newcommand{\Ccal}{\mathcal{C}}

\newcommand{\Hcal}{\mathcal{H}}

\newcommand{\Mcal}{\mathcal{M}}

\newcommand{\Tcal}{\mathcal{T}}

\newcommand{\diag}{\operatorname{diag}\,}

\newcommand{\SL}{\operatorname{SL}}

\let\leq\leqslant
\let\geq\geqslant

\begin{document}

\begin{center}
\begin{huge}
\begin{spacing}{1.0}
\textbf{Eisenstein series for $O(2, n+2)$}  
\end{spacing}
\end{huge}

\bigskip
by
\bigskip

\begin{large}
\textbf{Aloys Krieg\footnote{Aloys Krieg, Lehrstuhl A für Mathematik, RWTH Aachen University, D-52056 Aachen, krieg@rwth-aachen.de}}, \textbf{Hannah Römer\footnote{Hannah Römer, Dep. of Economics, University of Oxford, Manor Road Building, Manor Road, OX13UQ, United Kingdom, hannah.roemer@economics.ox.ax.uk}} and
\textbf{Felix Schaps\footnote{Felix Schaps, Lehrstuhl A für Mathematik, RWTH Aachen University, D-52056 Aachen, felix.schaps@matha.rwth-aachen.de}}
\end{large}
\vspace{0.5cm}\\
July 2023
\vspace{1cm}
\end{center}
\begin{abstract}
\textbf{Abstract.} We will characterize the Eisenstein series for $O(2,n+2)$ as a particular Hecke eigenform. As an application we show that it belongs to the associated Maaß space. If the underlying lattice is even and unimodular, this leads to an explicit formula of the Fourier coefficients.
\end{abstract}
\noindent\textbf{Classification MSC 2020: 11F55}
\vspace{2ex}\\

\newpage
\section{Introduction}

The orthogonal group $O(2, n+2)$ is associated with a Siegel domain of type IV. The attached spaces of modular forms have attracted a lot of attention, mainly by the additive lifts of Gritsenko [6] as well as the product expansions of Borcherds [1]. A lot of concrete examples were constructed by Wang and Williams [21], [22], [23].
We consider the classical Eisenstein series of Siegel type and show that it can be characterized as an eigenform under particular Hecke operators. Then we show that the associated Maaß space is also invariant under these Hecke operators. This allows us to demonstrate that the Eisenstein series belongs to the Maaß space. The proof generalizes the procedure for the paramodular group, i.e. $n = 1$ in our setting, in [11]. If we deal with an even unimodular lattice, the Fourier expansion can be calculated explicitly this way.

\section{Cusp forms}
\label{sect_2}

Throughout this paper we will assume 
\begin{gather*}\tag{1}\label{gl_1}
\begin{cases}
  & L=\Z ^n,  S\in\Z^{n\times n} \; \text{is positive definite and maximal even},\\
  & L_0 = \Z^{n+2}, \; S_0=\begin{pmatrix}
  0 & 0 & 1 \\ 0 & -S & 0 \\ 1 & 0 & 0
  \end{pmatrix}, \\
  & L_1=\Z ^{n+4}, \; S_1=\begin{pmatrix}
  0 & 0 & 1 \\ 0 & S_0 & 0 \\ 1 & 0 & 0
\end{pmatrix}\\
 \end{cases}
\end{gather*}
just as in [17].
Thus $S_1$ hat got the signature $(2, n+2)$. For matrices $A, B$ of appropriate size, we use the abbreviation
\[
A[B]:= B^{tr} AB.
\]
Consider the \emph{orthogonal group}\index{orthogonal group}
\[
O (T; R ):=\left\{ M\in R ^{n\times n}; T [M]= T\right\},
\]
holds for a symmetric non-degenerate $n\times n$ matrix $T$ and a subring $R$ of $\R$ as well as
\[
SO (T; R):= O(T; R)\cap\SL _n (R).
\]
Let $\Hcal$ denote the upper half-plane in $\C$ and $\Hcal _S$ the attached \emph{orthogonal half-space}\index{orthogonal half-space}
\[
\Hcal _S :=\left\{ w=\begin{pmatrix}
\tau \\ z \\ \tau '
\end{pmatrix} = u+iv\in\C ^{n+2}; \; S_0 [v]>0, \; \tau ,\; \tau '\in\Hcal\right\}.
\]
$O (S_1; \R )$ acts on $\Hcal _S\cup (-\Hcal _S)$ via
\begin{align*}
  & M = \begin{pmatrix}
  \alpha & a^{tr} & \beta \\ b & K & c \\ \gamma & d^{tr} & \delta
  \end{pmatrix}\in O (S_1; \R ), \; a,b,c,d\in\R ^{n+2}, \alpha ,\beta ,\gamma ,\delta\in\R ,\\
   & w \mapsto M\langle w\rangle := \left( -\frac{1}{2}S_0 [w]b +Kw+c\right) M\{ w\}^{-1}, \\
   & M\{ w\} :=-\frac{\gamma}{2} S_0 [w]+d^{tr}w+\delta .   
\end{align*}
$O^+ (S_1; \R )$ resp.  $SO^+ (S_1; \R )$ stands for the subgroup, which maps $\Hcal _S$ onto $\Hcal _S$. We consider the two standard embeddings of $\SL _2 (\R )$ into $SO^+ (S_1; \R )$.
Given a matrix $U=\begin{pmatrix}
\alpha & \beta \\ \gamma & \delta
\end{pmatrix}\in\SL _2 (\R )$ set
\begin{gather*}\tag{2}\label{gl_2}
U^{\downarrow}:= \begin{pmatrix}
HUH & 0 & 0 \\ 0 & I^{(n)} & 0 \\ 0 & 0 & U
\end{pmatrix}, \quad U^{\uparrow}:= \begin{pmatrix}
\alpha I^{(2)} & 0 & \beta H \\ 0 & I^{(n)} & 0 \\ \gamma H & 0 & \delta I^{(2)}
\end{pmatrix}\in SO^+ (S_1; \R ),
\end{gather*}
where $H=\begin{pmatrix}
-1 & 0 \\ 0 & 1
\end{pmatrix}$ and $I$ stands for the identity matrix of appropriate size. Denote by
\[
\Gamma _S:=\left\{ M\in SO^+ (S_1; \R ); \; M\in I+\Z ^{(n+4)\times (n+4)}S_1 \right\}
\]
the \textit{discriminant kernel}. The matrices

\begin{gather*}\tag{3}\label{gl_3}
  M_{\lambda} =\begin{pmatrix}
  1 & -\lambda ^{tr}S_0 & -\frac{1}{2}S_0 [\lambda ] \\
  0 & I & \lambda \\
  0 & 0 & 1
  \end{pmatrix}, \; \widetilde{M}_{\lambda}=\begin{pmatrix}
  1 & 0 & 0 \\
  \lambda & I & 0 \\
  -\frac{1}{2}S_0 [\lambda ] & -\lambda ^{tr}S_0 & 1
  \end{pmatrix}, \quad \lambda\in\Z ^{n+2},
\end{gather*}
belong to $\Gamma _S$ as well as

\begin{align*}\tag{4}\label{gl_4}
\begin{cases}
R_K = \begin{pmatrix}
1 & 0 & 0 \\
0 & K & 0 \\
0 & 0 & 1
\end{pmatrix}\in O^+ (S_1; \Z ) \;\; \text{for} \;\; K\in O^+ (S_0; \Z ),  \\ 
 e.g. \; K_{\mu} =\begin{pmatrix}
1 & \mu ^{tr}S & \frac{1}{2}S[\mu ]\\
0 & I^{(n)} & \mu \\ 
0 & 0 & 1
\end{pmatrix},  \; \tilde{K}_{\mu}=\begin{pmatrix}
1 & 0 & 0 \\ \mu & I^{(n)} & 0 \\ \frac{1}{2}S[\mu ] & \mu ^{tr}S & 1
\end{pmatrix}, \; \mu\in\Z ^n, \\
\hat{K} = \begin{pmatrix}
0 & 0 & 1 \\ 0 & I^{(n)} & 0 \\ 1 & 0 & 0
\end{pmatrix}, \quad \det \hat{K} = -1, \; \tilde{K}_{\mu}=\hat{K}K_{\mu }\hat{K}.
\end{cases}
\end{align*}
Given $k\in\N$ the space $\Mcal _k (\Gamma )$ for a group $\Gamma$,\; $\Gamma _S \subseteq \Gamma\subseteq O^+ (S_1; \Z )$,  of \textit{modular forms} of weight $k$ with respect to $\Gamma$ consists of all holomorphic functions $f: \Hcal _S\to \C$ satisfying 

\begin{gather*}\tag{5}\label{gl_5}
f(w) = f\underset{k}\mid M(w) := f(M\langle w\rangle ) M\{ w\}^{-k} \; \text{for all} \; M\in\Gamma .
\end{gather*}
Each such $f$ possesses a Fourier expansion of the form

\begin{gather*}\tag{6}\label{gl_6}
f(w) =\sum _{\lambda \in L_0^{\#}, \; \lambda\geq 0} \alpha _f (\lambda ) e^{2\pi i\lambda ^{tr} S_0 w},
\end{gather*}
where $L_0^{\#}= S_0^{-1}\Z ^{n+2}$ denotes the dual lattice and $\lambda = (\lambda _1, \ldots , \lambda _{n+2})^{tr}\geq 0$ means, $\lambda _1 \geq 0, \; \lambda _{n+2}\geq 0$ and $S_0 [\lambda ]\geq 0$. The subspace $\Ccal_k (\Gamma )$ of \textit{cusp forms} consists of all $f\in\Mcal _k (\Gamma )$ satisfying

\begin{gather*}\tag{7}\label{gl_7}
\alpha _f (\lambda )\neq 0 \quad \Rightarrow\quad S_0 [\lambda ] >0.
\end{gather*}

\begin{lemma}\label{lemma_1} %% Lemma 1
Let $(\ref{gl_1})$ hold and $0\neq y\in\R ^{n+2}, \; S_0 [y]=0$. Then 
\[
\left\{ v\in\R ^{n+2}; \; y^{tr} S_0v=0, \; v\geq 0\right\} \subseteq \left\{ v\in\R ^{n+2}; \; y^{tr}S_0 v=0, \; S_0 [v] =0\right\} = \R y.
\]
\end{lemma}

\begin{proof}
After an orthogonal transformation, one may assume $y=e_1 = (1, 0, \ldots , 0)^{tr}$, where the claim is straightforward.
\end{proof}

The Siegel $\phi$-operator is defined by
\begin{gather*}\tag{8}\label{gl_8}
f\mid\phi : \Hcal \to \C, \; \tau \mapsto\lim _{t\to\infty} f
\; \left( (it, 0,  \tau)^{tr}\right) ,
\end{gather*}
whenever this limit exists. Given $K=(u, \ldots , v)\in O^{+} (S_1, \Q )$ with $u\geq 0$ we have $S_0 [u]=0$ and $u^{tr}S_0 v=1$ and therefore 

\begin{gather*}\tag{9}\label{gl_9}
f\underset{k}\mid R_K\mid \phi (\tau )=\sum _{l\in\Q , lu\in L_0^{\#}\; lu\geq 0} \alpha _f (lu)e^{2\pi il\tau }
\end{gather*}
by Lemma \ref{lemma_1}. As (\ref{gl_1}) implies that any $\lambda\in L_0^{\#}$ with $S_0 [\lambda ]=0$ leads to $\lambda\in\Z ^{n+2}$, a direct consequence of (\ref{gl_9}) can be formulated as

\begin{lemma}\label{lemma_2} %% Lemma 2
Let $(\ref{gl_1})$ hold and $f\in\Mcal _k (\Gamma _S )$. Then $f$ is a cusp form if and only if
\[
f\underset{k}\mid R_K\mid \phi \equiv 0
\]
for all $K=(\lambda , \ast \ldots\ast )\in O^+ (S_1; \Q ), \; \lambda\in L_0^{\#}, \; \lambda\geq 0$ and $gcd (S_0\lambda )=1$.
\end{lemma}
Consider the maximal parabolic subgroup
\[
\Gamma _{S,\infty}:= \left\{ \begin{pmatrix}
\ast & \cdots & \ast & \ast \\ 
0 & \cdots & 0 & \pm 1
\end{pmatrix}\in\Gamma _S \right\}
\]
and denote the \textit{Eisenstein series} of even weight $k>n+2$ associated with $\Gamma _S$ by
\[
E_{k, S} (w) := \sum _{M:\Gamma _{S,\infty}\backslash \Gamma _S} M\{ w\}^{-k}.
\]
Moreover
\[
E_{k, ell} (\tau )=\frac{1}{2}\sum _{\substack{(\gamma , \delta )\in\Z\times \Z \\ gcd (\gamma , \delta )=1}} (\gamma\tau + \delta )^{-k}
\]
stands for the classical elliptic Eisenstein series.
\begin{lemma}\label{lemma_3} %% Lemma 3
Let $(\ref{gl_1})$ hold and let $k>n+2$ be even. Then 
\[
E_{k, S}\underset{k}\mid R_K \underset{k}\mid \phi (\tau ) =E_{k, ell} (\tau )
\]
holds for all $K=(u, \ldots , v)\in O^+ (S_1; \Q ), \; u\in L_0^{\#}, \; u\geq 0, \; gcd (S_0 u)=1$.
\end{lemma}
\begin{proof}
Note that the form of $M^{-1}$ for $M\in\Gamma _S$ leads to 
\[
M=\begin{pmatrix}
\ast & \ast & \ast \\ 
\gamma & d^{tr}S_0 & \delta
\end{pmatrix}, \quad d\in\Z^{n+2}.
\]
As $E_{k, S}(w)$ converges uniformly in vertical strips, we may interchange the limit and the summation. Note that
\[
\lim _{t\to\infty} M\left\{ (it, 0, \tau)^{tr}\right\} ^{-k} = \lim _{t\to\infty} \left( it \left( -\gamma \tau +d^{tr}S_0 u\right) +\left( \tau d^{tr}S_0 v+\delta\right)\right)^{-k} =0
\]
unless $\gamma =0$ and $d^{tr}S_0u =0$. According to Lemma \ref{lemma_1}, this means to sum over
\[
M=\begin{pmatrix}
\ast & \ast & \ast \\
0 & \gamma ' u ^{tr}S_0 & \delta
\end{pmatrix}, \; \gamma ' , \delta \in\Z , \; gcd (\gamma ' , \delta )=1.
\]
Hence we get the claim.
\end{proof}

\section{Hecke-Theory}
\label{sect_2}

Given $r\in\N$ let
\begin{gather*}
\Delta _S (r) := \left\{ R\in\Z ^{(n+4)\times (n+4)}; \; \frac{1}{\sqrt{r}}R\in SO^+ (S_1; \R )\right\}, \quad \Delta _S := \bigcup _{r=1}^{\infty} \Delta _S (r).
\end{gather*}
It is well-known (cf. [7]) that $(\Gamma _S, \Delta _S)$ is a Hecke pair. From [19], Theorem 2, we quote
\begin{lemma}\label{lemma_4} %% Lemma 4
Let $(\ref{gl_1})$ hold and $R\in\Delta _S (r), \; r\in\N$. \\
a) \; The right coset $\Gamma _S R$ contains a representative 
\[
\begin{pmatrix}
\alpha & \ast & \ast \\
0 & \ast & \ast \\
0 & 0 & \delta
\end{pmatrix}, \quad \alpha , \delta\in\N , \quad \alpha\delta =r,
\]
where $\alpha$ is the $gcd$ of the entries of the first column of $R$. \\
b)\; The double coset $\Gamma _SR\Gamma _S$ contains a representative 
\[
R^*=\begin{pmatrix}
\alpha ^* & 0 & 0 \\ 
0 & K^* & 0 \\ 
0 & 0 & \delta ^*
\end{pmatrix}, \quad \alpha ^*, \;\delta ^*\in\N, \; \alpha ^*\delta ^*=r,
\]
where $\alpha ^*$ is the $gcd$ of the entries of $R$ and $R^*$. \\
\end{lemma} 
Given $f\in\Mcal _k (\Gamma _S)$ and $R\in\Delta _S$ we define the \textit{Hecke operator}
\begin{gather*}\tag{10}\label{gl_10}
f\underset{k}\mid \Gamma _S R\Gamma _S:=\sum _{M:\Gamma _S\backslash\Gamma _S R\Gamma _S} f\underset{k}\mid M\in\Mcal _k(\Gamma _S)
\end{gather*}
without any normalizing factor. A direct consequence of (\ref{gl_7}) and Lemma \ref{lemma_4} a) is
\begin{gather*}\tag{11}\label{gl_11}
f\in \Ccal _k (\Gamma _S)\quad \Rightarrow\quad f\underset{k}\mid \Gamma _S R \Gamma _S\in \Ccal _k (\Gamma _S).
\end{gather*}

\begin{lemma}\label{lemma_5} %% Lemma 5
Let $(\ref{gl_1})$ hold and $k>n+2$ be even. Then the Eisenstein series $E_{k,S}$ is a simultaneous Hecke eigenform. More precisely for $R\in\Delta _S$
\[
E_{k, S} \underset{k}\mid \Gamma _S R\Gamma _S =\rho _k E_{k,  S},
\]
holds, where
\[
\rho _k:= \sum_{\begin{pmatrix}
\ast & \ldots & \ast & \ast \\ 
0 & \ldots & 0 & \delta
\end{pmatrix}: \Gamma _S\backslash \Gamma _S R \Gamma _S} \delta ^{-k}.
\]
\end{lemma}
\begin{proof}
Proceed in the same way as Freitag [4], IV.4.7. The eigenvalue $\rho _k$ arises from the calculation of the constant Fourier coefficient.
\end{proof}
Note that $\rho _0=\# (\Gamma _S\backslash \Gamma _S R \Gamma _S$). As 
\[
\Hcal\to\R, \quad w\mapsto S_0[v]^{k/2} |f(w)|,
\]
is bounded for $f\in \Ccal _k (\Gamma _S)$, we conclude
\begin{lemma}\label{lemma_6} %% Lemma 6
Let $(\ref{gl_1})$ hold and $0\neq f\in \Ccal _k(\Gamma _S)$ such that
\[
f\underset{k}\mid \Gamma _S R\Gamma _S =\rho f
\]
for $R\in\Delta _S (r)$ and some $\rho\in\C$. Then
\[
|\rho |\leq r^{-k/2}\rho _0.
\]
\end{lemma}

\begin{proof}
Confer Freitag [4], IV.4.8.
\end{proof}

Given a prime $q$ we consider the simplest non-trivial Hecke operator
\begin{gather*}
\Tcal _S (q) := \Gamma _S\diag (1, q, \ldots , q, q^2)\Gamma _S = \{ R\in\Delta _S (q^2); r_q (R)=1\} ,
\end{gather*}
where $r_q$ denotes the rank over $\Z / q\Z$
\begin{lemma}\label{lemma_7} %% Lemma 7
Let $(\ref{gl_1})$ hold and $q$ be a prime, $q\nmid\det S$. Then a set of representatives of $\Gamma _S\backslash \Tcal _S (q)$ is given by 
\begin{itemize}
\item[a)] $\diag (1, q, \ldots , q, q^2)M_{\lambda }, \; \lambda :\Z^{n+2}\mod q$, 
\item[b)] $\diag (q^2, q, \ldots , q, 1)$, 
\item[c)] $\diag (q, 1, q, \ldots ,q, q^2, q) K_{\nu} M_{\lambda ^*}, \; \nu :\Z ^n\mod q, \; \lambda ^*=le_{1}, \; l:\Z\mod q$, 
\item[d)]  $\diag (q, q^2, q, \ldots , q, 1, q) M_{\lambda ^*}, \; \lambda ^* = le_{n+2}, \; l:\Z\mod q$, 
\item[e)]  $\diag (q, q^2, q, \ldots , q, 1, q) \tilde{K}_{\mu} M_{\lambda ^*}, \; \lambda ^*=le_{n+2}, \; l: \Z\mod q, \; \mu : \Z ^n\mod q, \; $  \\
 $\mu\not\equiv 0\mod q,\frac{1}{2}S[\mu ]\equiv 0\mod q $,  
\item[f)]  $\diag (q, q^2, q, \ldots , q, 1, q) \tilde{K}_{\mu _j}K_{re_j}M_{\lambda ^*}, \; \lambda ^*=le_{n+2}, \; l:\Z\mod q, \; \mu _j:\Z ^n\mod q, $ \\
$\mu _j^{tr} S\equiv (\zeta _1, \ldots , \zeta _j, 0, \ldots , 0 )\mod q,  \; r \;\; \text{fixed,} \;\;  r\not\equiv 0\mod q,\; r\zeta _j\equiv -1\mod q, $ \\
$\frac{1}{2} S[\mu _j]\equiv 0\mod q, \; j=1, \ldots , n$.
\end{itemize}
\end{lemma} 
\begin{proof}
Apply Lemma \ref{lemma_4} a), which yields $\alpha\delta =q^2$. The cases $\delta =1$ and $\delta =q^2$ are clear by virtue of $r_q (M)=1$. Hence we are left with
\[
M=\begin{pmatrix}
q & 0 & 0 \\
0 & K & 0 \\
0 & 0 & q
\end{pmatrix} M_{\lambda ^*}, \; S_0 [K] = q^2 S_0, \; r_q (K)=1, \; K\lambda ^*:K\Z^{n+2}\mod q.
\]
Let $\sigma$ be the $gcd$ of the first column of $K$. If $\sigma =1$ or $\sigma=q^2$, we apply the above procedure to $MJ^{\downarrow}$ (cf. (2)). If we multiply by $J^{\downarrow}$ from the left, we obtain a representative in $\Gamma _SM$ of the form
\[
\begin{pmatrix}
q & 0 & 0 \\ 0 & K & 0 \\ 0 & 0 & q
\end{pmatrix} M_{\lambda ^*}, \; K=\begin{pmatrix}
1 & 0 & 0 \\ 0 & qI & 0 \\ 0 & 0 & q^2
\end{pmatrix}K_{\nu}, \; \nu :\Z ^n\mod q, \; K=\begin{pmatrix}
q^2 & 0 & 0 \\
0 & qI & 0 \\
0 & 0 & 1
\end{pmatrix}, \; I=I^{(n)}.
\]
Hence we are left with $\sigma = q$. Then $\sigma '$, the $gcd$ of the last column of $K$, is $1$ or $q$. If $\sigma '=1$, we apply the same procedure as above to $\hat{K}K\hat{K}$ and obtain a form
\[
\begin{pmatrix}
q^2 & 0 & 0 \\
0 & qI & 0 \\
0 & 0 & 1
\end{pmatrix} \tilde{K}_\mu =\begin{pmatrix}
q^2 & 0 & 0 \\
q\mu & qI & 0 \\
\frac{1}{2}S[\mu ] & \mu ^{tr}S & 1
\end{pmatrix}
\]
with $\mu:\Z ^n\mod q, \;\; \mu\not\equiv 0\mod q, \; \; \frac{1}{2}S[\mu ]\equiv 0\mod q$ due to $\sigma = q$.
If $\sigma '=q$ we chose $1\leq j\leq n$ maximal such that the $(j+1)$th column of $K$ is $\not\equiv 0\mod q$. Hence we can apply the same procedure as above to $\hat{K}K K_{e_j}\hat{K}$ and obtain a representative
\[
\begin{pmatrix}
q^2 & 0 & 0 \\ 0 & qI & 0 \\ 0 &0 & 1
\end{pmatrix} \tilde{K}_{\mu _j}K_{e_j}= \begin{pmatrix}
q^2 & q^2e_j^{tr}S & q^2\frac{1}{2}S[e_j] \\
q\mu _j & qI + q\mu _je_j^{tr}S & q(e_j+\frac{1}{2}S[e_j]\mu _j) \\
\frac{1}{2}S[\mu _j] & (\mu _j + \frac{1}{2}S[\mu _j]e_j)^{tr}S & \frac{1}{2}S[\mu _j ]\frac{1}{2}S[e_j ]+\mu _j^{tr}Se_j +1 
\end{pmatrix}
\]
where $\mu _j:\Z ^n\mod q, \;  \frac{1}{2}S[\mu _j ]\equiv 0\mod q, \;  \mu _j^{tr}S\equiv (\zeta _1, \ldots , \zeta _j, 0\ldots , 0)\mod q$, and $\zeta _j=\mu _j^{tr}Se_j\equiv -1\mod q$ in order to satisfy the assumptions on the $gcd$ of the columns.  
Clearly these right cosets are mutually distinct.
\end{proof}

We describe an application, which generalizes [3].

\begin{theorem}\label{theorem_1} %% Theorem 1
Let $(\ref{gl_1})$ hold and $k>n+2$ be even. Let $f\in\Mcal _k(\Gamma _S)$ satisfy
\[
f\underset{k}\mid R_K\mid \phi (\tau )=E_{k, ell} (\tau )
\]
for all $K=(u, \ldots , v)\in O^+ (S_1, \Q ), \; u\in L_0^{\#}, \; gcd (S_0 u)=1$. If $q$ is a prime, $q\nmid \det S$ such that 
\[
f\underset{k}\mid \Tcal _S (q) = \rho f
\]
for some $\rho\in\C$, then
\[
f=E_{k, S}.
\]
\end{theorem}

\begin{proof}
Due to Lemma \ref{lemma_2} and \ref{lemma_3} $g:= f-E_{k, S}$ is a cusp form. As $\alpha _f (0)\neq 0$ we obtain $\rho = \rho _k$ from Lemma \ref{lemma_5}. Then Lemma \ref{lemma_6} yields
\[
\rho _k\leq q^{-k}\rho_0
\]
If $N$ denote the number of right cosets in Lemma \ref{lemma_7} c)-f), this means 
\[
1+q^{n+2-2k}+q^{-k}N\leq q^{-k} (1+q^{n+2} +N)
\]
In view of $k>n+2$, this is a contradition and $g\equiv 0$ follows.
\end{proof}

\begin{remark}\label{remark_1} %% Remark 1
a) \; The proof of Lemma $7$ shows for a prime $q$ that
\begin{align*}
\Tcal _S^* (q) & = \{ R\in\Delta _S (q^2); r_q (R) = 2\} \\
& = \Gamma _S\diag (1, 1, q, \ldots , q, q^2, q^2) \Gamma _S.
\end{align*}
b) \; If we consider the isomorphism of $\Gamma _S$ with the Siegel resp. the Hermitation modular group $\Gamma$ of degree $2$ in [9] resp. [18], we conclude that $\Tcal _S (q)$ resp. $\Tcal _S^* (q)$ corresponds to
\[
\Gamma \diag (1, 1, q, q)\Gamma \quad \text{resp.} \quad \Gamma \diag (1, q, q^2, q) \Gamma .
\]
c) \; The Hecke operators $\Tcal _S (p), \; \Tcal _S (q)$ for primes $p$ and $q$ commute and are self-adjoint with respect to the Petersson scalar product. \\
d)\; If we consider $q|\det S$, there may appear more right cosets than quoted in Lemma \ref{lemma_7}. This is already true in the case $n=1$, where we deal with the paramodular group of degree \ref{gl_2} (cf. [5]). \\
e) \; It is possible to derive Theorem \ref{theorem_1} under the weaker assumption $\alpha _f (0)=1$. In this case one has to assume that $f$ is an eigenform under infinitely many $q$. \\
f) \; As representatives of $\Gamma _S\backslash O^+ (S_1; \Z )$ may be chosen in the form $R_K$ in (\ref{gl_4}), we conclude
\[
E_{k, S}\in\Mcal _k (O^+ (S_1; \Z )) \quad \text{for even} \quad k>n+2.
\] 
\end{remark}

\section{The Maaß space}
\label{sect_4}

$f\in\Mcal _k (\Gamma _S)$ belongs to the \textit{Maaß space} $\Mcal _k^* (\Gamma _S)$ if

\begin{gather*}\tag{11}\label{gl_11}
\alpha _f (\lambda )=\sum _{d\mid \epsilon (\lambda )} d^{k-1} \alpha _f \begin{pmatrix}
1 \\ \mu /d \\ lm / d^2
\end{pmatrix} \quad \text{for all} \;\; 0\neq \lambda=\begin{pmatrix}
m \\ \mu \\ l
\end{pmatrix}\in L_0^{\#},
\end{gather*}
where $\epsilon (\lambda ):= gcd (S_0 \lambda )$. 
Considering Lemma \ref{lemma_2} and the Fourier expansion of the elliptic Eisenstein series (cf. [14]), we note that any $f\in\Mcal _k^* (\Gamma _S)$ and any $K=(\lambda , *,\ldots ,*)\in O^+ (S_1; Q), \; \lambda\in L_0^{\#}, \; \epsilon (S_0\lambda )=1$ satisfy
\begin{gather*}\tag{12}\label{gl_12}
f\underset{k}\mid R_K\mid \phi (\tau )=\alpha _f (0) E_{k, ell} (\tau ), \; \frac{-2k}{B_k}\alpha _f (0)=\alpha _f (\lambda ).
\end{gather*}
We define the $p$-\textit{Maaß condition}\index{Maaß condition} for a prime $p$ by
\begin{gather*}\tag{13}\label{gl_13}
\begin{cases}
  & \alpha _f \begin{pmatrix}
  p^rm \\ \mu \\ l
  \end{pmatrix}= \sum\limits_{j=0}^r p^{j(k-1)} \alpha _f\begin{pmatrix}
  m \\ p^{-j}\mu \\ p^{r-2j}\lambda
  \end{pmatrix}, \\
  & \alpha _f \begin{pmatrix}
  0 \\ 0 \\ p^rm
  \end{pmatrix}= \left( \sum\limits_{j=0}^r p^{j(k-1)}\right)\alpha _f \begin{pmatrix}
  0 \\ 0 \\ m
  \end{pmatrix}
 \end{cases}
\end{gather*}
for all such $\lambda$ with $p\nmid m$, where we set $\alpha _f (\lambda )=0$ if $\lambda\not\in L_0^{\#}$.
Instead of the symmetric relation in [12], we are going to describe the Maaß condition via Hecke operators associated with the minimal parabolic subgroup
\begin{align*}
\Gamma _S^P & =\left\{ M\in\Gamma _S; \; M-I \; \text{strictly upper triangular} \right\} \\
& = \left\{ K_{\mu}M_{\lambda}; \; \mu\in\Z ^n, \; \lambda\in\Z^{n+2}\right\} .
\end{align*}
Now we embed the classical elliptic Hecke operators via (\ref{gl_2}) into the orthogonal group. For a prime $p$ set
\begin{gather*}\tag{14}\label{gl_14}
f\underset{k}\mid T_p^{\downarrow}:= f\underset{k}\mid \diag (p^2, p^2, p, \ldots , p,  1)+\sum _{l\mod p} f\underset{k}\mid \diag (p, p^2, p, \ldots , p, 1, p) M_{le_{n+2}},
\end{gather*}
\begin{gather*}\tag{15}\label{gl_15}
f\underset{k}\mid T_p^{\uparrow}:= f\underset{k}\mid \diag (p^2, p^2, p, \ldots , p, 1, 1)+\sum _{l\mod p}f\underset{k}\mid (pI)\cdot M_{l/p e_1}.
\end{gather*}

\begin{theorem}\label{theorem_2} %% Theorem 2
Let $(\ref{gl_1})$ hold and $f\in\Mcal _k (\Gamma _S)$ for an even $k$. Then the following assertions are equivalent.  
\begin{itemize}
\item[(i)] $f$ belongs to the Maaß space. 
\item[(ii)] $f$ satisfies the $p$-Maaß condition for every prime $p$. 
\item[(iii)]$\alpha _f \begin{pmatrix}
pm \\ \mu \\ l
\end{pmatrix} +p^{k-1} \alpha _f \begin{pmatrix}
m/p \\ \mu /p \\ l
\end{pmatrix} = \alpha _f\begin{pmatrix}
m \\ \mu \\ pl
\end{pmatrix} + p^{k-1} \alpha _f \alpha _f \begin{pmatrix}
m \\ \mu /p \\ l/ p
\end{pmatrix}$ for all $m, l, \mu$ and all primes $p$. 
\item[(iv)]$f\underset{k}\mid T_p^{\uparrow} = f\underset{k}\mid T_p^{\downarrow}$ for all primes $p$. 
\end{itemize}
\end{theorem}

\begin{proof} 
(i) $\Leftrightarrow$ (ii) This is obvious from (\ref{gl_11}) and (\ref{gl_13}). \\
(ii) $\Rightarrow$ (iii) Merely apply (\ref{gl_13}) to both sides of $(iii)$. \\
(iii) $\Rightarrow$ (ii)  Use an induction on $r$. \\
 (iii) $\Leftrightarrow$ (iv) Compute the Fourier coefficient of $\begin{pmatrix}
  m \\ \mu \\ pl
  \end{pmatrix}$ for $f\mid T_p^{\uparrow}$ and $f\mid T_p^{\downarrow}$.  Then the claim follows.
\end{proof}
A sightly weaker assumption is dealt with in

\begin{corollary}\label{corollary_1} %% Corollary 1
Let $(\ref{gl_1})$ hold, $f\in\Mcal _k (\Gamma _S)$, $k$ even and $q$ be a prime. Then the following assertions are equivalent: \\
(i) \;$ f\in\Mcal ^*_k (\Gamma _S)$. \\
(ii)\; $f$ satisfies the $p$-Maaß condition for every prime $p\neq q$.
\end{corollary}

\begin{proof}
Let $f$ satisfy $(ii)$ and let $f^M$ be the Maaß lift of the first Fourier-Jacobi coefficient of $f$. Consider
\[
g:= f-f^M\in\Mcal _k (\Gamma _S).
\]
Due to Theorem \ref{theorem_1} $g$ possesses a Fourier expansion of the form
\[
g(w)=\sum _{\lambda=\begin{pmatrix}
m \\ \mu \\ pl
\end{pmatrix}\in L_0^{\#}, \lambda \geq 0} \alpha _g (\lambda )e^{2\pi i\lambda^{tr}S_0 w}.
\]
Hence we additionally obtain
\[
g(w+\frac{1}{p}e_{1})=g(w).
\]
Then $g\equiv 0$ follows from [19], because $SO ^+(S_1; \Z )$ is the unique maximal discrete extension of $\Gamma _S$ in $SO^+ (S_1; \R )$.
\end{proof}
The next application concerns the Hecke operator $\Tcal _S (q)$.

\begin{corollary}\label{corollary_2} %% Corollary 2
Let $(\ref{gl_1})$ hold. Then the Maaß space is invariant under all Hecke operators $\Tcal _S (q), q$ prime, $q\nmid\det S$.
\end{corollary}
\begin{proof}
We apply Corollary \ref{corollary_1} and Theorem \ref{theorem_2}. The claim follows from commutation relations in Hecke algebras for primes $p, q, p\neq q$, by embedding them into the Hecke algebras with respect to minmal parabolic subgroups (cf. [15]). Therefore we choose the entries of the off-diagonal elements of matrices in Lemma \ref{lemma_7} divisible by $p^2$ and of the double cosets in (\ref{gl_14}), (\ref{gl_15}) by $q^2$. First of all we interprete (\ref{gl_14}) and (\ref{gl_15}) as Hecke operators with respect to $\Gamma _S^P$ as well as $\Gamma _S^P\diag (1, q, \ldots ,q, q^2)\Gamma _S^P$, $\Gamma _S^P\diag (q, 1, q, \ldots , q, q^2, q)\Gamma _S^P$, $ \Gamma _S^P (q^2, q, \ldots , q, 1)\Gamma _S^P$ and
 $\Gamma _S^P (q, q^2, q, \ldots , q, 1, q) \Gamma _S^P$ in Lemma \ref{lemma_7} a)-d).  We can show that they commute. Then we interprete (\ref{gl_14}) and (\ref{gl_15}) as Hecke operators with respect to
\begin{gather*}
\hat{K} \Gamma _S^P\hat{K} = \left\{ \tilde{K}_{\mu} M_{\lambda}; \; \mu\in\Z ^n, \; \lambda\in\Z ^{n+2}\right\}.
\end{gather*}
One calculates that the sum of these double cosets commutes with the sum of the $\hat{K}\Gamma _S^P\hat{K}$ double cosets of the matrices arising in Lemma \ref{lemma_7} e), f).
\end{proof}
As an application we obtain
\begin{corollary}\label{corollary_3} %% Corollary 3
Let $(\ref{gl_1})$ and $k>n+2$ be even. Then
\[
E_{k,S}\in\Mcal _k^* (\Gamma _S).
\]
\end{corollary}
\begin{proof}
$\Mcal _k^* (\Gamma _S)$ contains a non-cusp form, e.g. the Maaß lift of the first Fourier-Jacobi coefficient of $E_{k,S}$. Let $f$ be an eigenform under $\Tcal _S (q)$ for some prime $q\nmid \det S$. Then (\ref{gl_12}) yields
\[
f\underset{k}\mid R_K \underset{k}\mid \phi (\tau )=\alpha _f (0) E_{k, ell} (\tau )
\]
for all $K\in O^+ (S_1; \Q), \; K=(u, \ast, \ldots , \ast ), \; u\in L_0^{\#}, \; u\geq 0, \; \epsilon (u)=1$.
Thus $f = E_{k, S}$ follows from Theorem \ref{theorem_1}.
\end{proof}

\begin{remark}\label{remark_2} %% Remark 2
a) \; Any $f\in\Mcal ^*_k (\Gamma _S)$ is also invariant under $R_{\hat{K}}$ with $\hat{K}$ in (4).\\
b) \; The same proof holds for the Maaß space in $\Mcal _k (\Gamma )$ for an arbitrary subgroup $\Gamma , \Gamma _S\subseteq \Gamma \subseteq O^+ (S_1; \Z )$. 
\end{remark}

\section{Even unimodular lattices}
\label{sect_5}
If (\ref{gl_1}) holds with $\det S=1$ we know from [16] that the map
\begin{gather*}\tag{16}\label{gl_16}
\Mcal _k^* (\Gamma _S)\to\Mcal _{k-n/2}(\Gamma ), \quad f\mapsto f^* (\tau )=\sum _{l=0}^{\infty}\alpha _f \left( (l,  0,  1)^{tr}\right) e^{2\pi il\tau},
\end{gather*}
is an isomorphism, where $\Mcal _k (\Gamma )$ stands for the space of elliptic modular forms of weight $k$.  We use the notation of Bernoulli numbers $B_k$ from [14]. 
\begin{corollary}\label{corollary_4} %% Corollary 4
Let $(\ref{gl_1})$ hold with $\det S=1$ and $k=\frac{n}{2}$ or $k>\frac{n}{2}+2$ be even. Then there exists an $F_{k, S}\in\Mcal _k^* (\Gamma _S)$ satisfying
\[
F_{k, S}^* (\tau) = -\frac{2k}{B_k} E_{k-n/2, ell} (\tau ).
\]
the Fourier coefficients $\alpha _k (\lambda ), \lambda\in L_0^{\#},$ of $F_{k, S}$ are given by
\begin{gather*}\tag{17}\label{gl_17}
\alpha _k(\lambda )= 
\begin{cases}
 1, \quad & \text{if} \; \lambda =0,  \\
 -\frac{2k}{B_k}\sigma _{k-1}(\epsilon (\lambda )), \quad & \text{if}\;  \lambda\neq 0, S_0[\lambda ]=0,  \\
 \frac{2k(2k-n)}{B_k B_{k-n/2}} \sum _{d\mid \epsilon(\lambda )}d^{k-1}\sigma _{k-1-n/2} (S_0 [\lambda ]/2d), \quad &  \text{if}\; S_0[\lambda ]>0, k>\frac{n}{2},\\
 0, \quad & \text{if} \; S_0[\lambda ]>0, \; k=\frac{n}{2}.
\end{cases}
\end{gather*}
\end{corollary}
Let $\Tcal (m)$ stand for the elliptic Hecke operator without any normalizing factor.
\begin{theorem}\label{theorem_3} %% Theorem 3
Let $(\ref{gl_1})$ hold with $\det S=1$. Then $f\in\Mcal _k^* (\Gamma _S)$ is an eigenform under $\Tcal _S (q)$ for a prime $q\nmid \det S$ if and only if $f^*$ is an eigenform under $\Tcal (q^2)$. 
\end{theorem}
\begin{proof}
Let $g=f\underset{k}\mid \Tcal _S (q)$ and set $\alpha _f(\lambda )=0$ for $\lambda\not\in L_0^{\#}, \; \lambda=\begin{pmatrix}
l \\ \mu \\ m
\end{pmatrix}$. It follows from Lemma \ref{lemma_7} that
\begin{gather*}
\alpha _g (\lambda )=  \alpha _f \left( \frac{1}{q}\lambda\right)+q^{n+2-2k}\alpha _f(q\lambda )+ \sum _K q^{1-k}\alpha _f\left( \frac{1}{q}K\lambda \right), 
\end{gather*}
where we sum over $K$ in Lemma \ref{lemma_7} c) - f). If $\lambda = (l,  0,  1) ^{tr}$, note that
\[
\begin{pmatrix}
q & 0 & 0 \\ 0 & I & 0 \\ 0 & 0 & \frac{1}{q}
\end{pmatrix} \tilde{K} _{\mu}\lambda\not\in L_0^{\#}, \; \lambda ^* = \begin{pmatrix}
q & 0 & 0 \\ 0 & I & 0 \\ 0 & 0 & \frac{1}{q}
\end{pmatrix} \tilde{K}_{\mu _j}K_{e_j} \lambda\in L_0^{\#}
\]
with $\epsilon (S_0\lambda ^*)=1$, hence $\alpha _g(\lambda ^*) = \alpha _f(\lambda )$. If $N'$ denotes the number of right cosets in Lemma \ref{lemma_7} f) we obtain 
\[
\alpha _g(\lambda )=q^{n+2-2k}\alpha _f \begin{pmatrix}
ql \\ 0 \\ q
\end{pmatrix}+q^{1-k}\sum _{\nu :\Z ^n/q\Z ^n}\alpha _f\begin{pmatrix}
(l+\frac{1}{2}S[\nu ])/q \\ \nu \\ q
\end{pmatrix}+ N'q^{-k}\alpha _f\begin{pmatrix}
l \\ 0 \\ 1
\end{pmatrix}.
\]
 Observe that (cf. [14]).
\[
\# \{\nu\mod q; \; \frac{1}{2}S[\nu ]\equiv -l\mod q\}=
\begin{cases}
   q^{n-1}-q^{n/2-1} \quad & \text{if}\; l\not\equiv 0\;\text{mod}\; q,\\
   q^{n-1}-q^{n/2-1}+q^{n/2} \quad & \text{if}\; l\equiv 0 \;\text{mod}\; q. \\
 \end{cases}
\]
Then the Maaß condition (\ref{gl_11}) leads to
\begin{align*}
\alpha _g (\lambda )= & q^{n+2-2k}\alpha _f\begin{pmatrix}
q^2l \\ 0 \\ 1
\end{pmatrix}+\alpha _f \begin{pmatrix}
l/q^2 \\ 0 \\ 1
\end{pmatrix}+q^{-k}\left( q^{n+1}+q^{n}-q^{n/2}+N'\right)\alpha _f\begin{pmatrix}
l \\ 0 \\ 1
\end{pmatrix} \\
& + \begin{cases}
 0 \; & \text{if} \; q\nmid l \\
 q^{\frac{n}{2}+1-k}\alpha _f\begin{pmatrix}
l \\ 0 \\ 1
\end{pmatrix} \; & \text{if} \; q\mid l.
\end{cases}
\end{align*}
The classical Hecke theory therefore shows that
\[
g^* = f^*\underset{k-n/2}\mid \Tcal (q^2) + \left(q^{n+1}+q^{n}-q^{n/2}+N'\right) q^{-k} f^*.
\]
This yields the claim as (\ref{gl_16}) is an isomorphism. 
\end{proof}
As $E_{k, ell} (\tau )$ is an eigenform under all Hecke operators, Theorem \ref{theorem_1} and Theorem \ref{theorem_3} imply
\begin{corollary}\label{corollary_5} %% Corollary 5
Let $(\ref{gl_1})$ hold with $\det S=1$ and $k>n+2$ be even. Then
\[
E_{k, S}=F_{k, S}\in\Mcal _k^* (\Gamma _S).
\]
\end{corollary}
\begin{remark}\label{remark_3} %% Remark 3
a) \; Corollary 5 was already proved by Woitalla [24].  For the general version compare [20].\\
b) \; It is conjectured that the modular forms $F_{k, S}, \; k=\frac{n}{2}$ or $\frac{n}{2}+2<k\leq n+2$ arise from $E_{k,S}$ via the Hecke trick.
If $k=\frac{n}{2}$ we obtain a singular modular form in $\Mcal _{n/2}^* (\Gamma _S)$ (cf. [13]). \\
c)\; It is now clear the modular forms $E_k$ in [2] defined as Maaß lifts are actually the Eisenstein series for $\Gamma _S$, where $S$ is the $E_8$ lattice \\
d)\; In [9] the authors investigated the relation $E_4^2=E_8$ for Hermitian Eisenstein series of degree 2. If we consider such a relation for $\det S=1$, a comparison of Theorem \ref{theorem_1} and (\ref{gl_16}) shows that it is only possible in the case $n=8$. In this case the relations
\[
F_{4, S}^2=F_{8, S} \quad \text{und} \quad F_{4,S}\cdot F_{10, S}=E_{14, S}
\]
are consequences of [2]. Then Corollary 4 yields
\[
2\sigma _3 \left( \frac{1}{2} S_0[\lambda ]\right) =\sum _{\substack{\nu , \mu\in\Z ^{10}\backslash \{ 0\}, \nu\geq 0, \mu\geq 0 \\ \nu +\mu =\lambda , \; S_0 [\nu ]=S_0 [\mu ] =0}} \sigma _3 \left( \epsilon (\nu  )\right) \sigma _3 \left( \epsilon (\mu )\right)
\]
for all $\lambda\in\Z ^{10}, \; \lambda >0$.
\end{remark}
The authors thank E. Freitag for pointing out Remark 3d) to them.
\vspace{6ex}

%============================================

\bibliography{bibliography_krieg_2021} 
\bibliographystyle{plain}

\end{document}